\documentclass[11pt]{amsart}
\usepackage{amsmath,amsthm,amsfonts,amssymb,mathrsfs}
\usepackage{color,hyperref}

\usepackage[all]{xy}

\newcommand{\Fin}{\mathcal{F}\kern-1pt\mathit{in}}
\newcommand{\w}{\omega}

\newcommand{\F}{\mathcal F}

\newcommand{\U}{\mathcal U}
\newcommand{\V}{\mathcal V}

\newcommand{\C}{\mathcal C}
\newcommand{\Ra}{\Rightarrow}

\newcommand{\dom}{\mathrm{dom}}

\newcommand{\Tau}{\mathcal T}

\newcommand{\Ord}{\mathsf{T\!sL}_\Omega}

\newcommand{\TsLone}{\mathsf{T}_{\!\!\mathsf{1}\!}\mathsf{sL}}
\newcommand{\sTsLone}{\mathsf{s\!T}_{\mathsf{\!\!1\!}}\mathsf{sL}}
\newcommand{\TsL}{\mathsf{T\!sL}}
\newcommand{\sTsL}{\mathsf{s\!T\!sL}}
\newcommand{\TsLw}{\mathsf{T\!sL}_\omega}

\newcommand{\TS}{\mathsf{T\!S}}

\newcommand{\eC}{\mathsf{e}{:}\mathcal C}
\newcommand{\hC}{\mathsf{h}{:}\mathcal C}
\newcommand{\peC}{\mathsf{pe}{:}\mathcal C}
\newcommand{\pC}{\mathsf{p}{:}\mathcal C}

\newcommand{\eTS}{\mathsf{e}{:}\!\mathsf{TS}}
\newcommand{\hTS}{\mathsf{h}{:}\!\mathsf{TS}}

\newtheorem{theorem}{Theorem}[section]

\newtheorem{lemma}[theorem]{Lemma}
\newtheorem{corollary}[theorem]{Corollary}
\newtheorem{example}[theorem]{Example}

\theoremstyle{definition}
\newtheorem{definition}[theorem]{Definition}
\newtheorem{remark}[theorem]{Remark}
\newtheorem{problem}[theorem]{Problem}

\title[Characterizing chain-compact and chain-finite topological semilattices]{Characterizing chain-compact and chain-finite\\ topological semilattices}
\author{Taras Banakh and Serhii Bardyla}
\address{T.~Banakh: Ivan Franko National University of Lviv
(Ukraine) and\newline\indent Jan Kochanowski University in Kielce
(Poland)}
\email{t.o.banakh@gmail.com}
\address{S.~Bardyla: Ivan Franko National University of Lviv, Universytetska 1, 79000, Lviv, Ukraine}
\email{sbardyla@yahoo.com}

\begin{document}

\begin{abstract} In the paper we present various characterizations of chain-compact and chain-finite topological semilattices. A topological semilattice $X$ is called {\em chain-compact} (resp. {\em chain-finite\/}) if each closed chain in $X$ is compact (finite).
In particular, we prove that a (Hausdorff) $T_1$-topological semilattice $X$ is chain-finite (chain-compact) if and only if for any closed subsemilattice $Z\subset X$ and any continuous homomorphism $h:Z\to Y$ to a (Hausdorff) $T_1$-topological semilattice $Y$ the image $h(Z)$ is closed in $Y$.
 \end{abstract}

 \maketitle

\section{Introduction} It is well-known
that a topological group $X$ is Raikov-complete (i.e., complete in its two-sided uniformity) if and only if for any isomorphic topological embedding $h:X\to Y$ into a Hausdorff topological group $Y$ the image $h(X)$ is closed in $Y$. This classical result was extended by Bardyla, Gutik and Ravsky \cite{BGR} who proved that for any isomorphic topological embedding $h:X\to Y$ of Raikov-complete topological group into a Hausdorff quasi-topological group $Y$ the image $h(X)$ is closed in $Y$. Those results yield external characterizations of the Raikov-completeness of topological groups.
This external approach can be used to define some completeness properties in the realm of topologized semigroups where the standard technique of uniformities does not work.

By a {\em topologized semigroup} we understand a semigroup $S$ endowed with a topology. If the semigroup operation $S\times S\to S$ is (separately) continuous, then $S$ is called a ({\em semi}){\em topological semigroup}. 


Topologized semigroups are objects of many categories, which differ by morphisms. The smallest is the category whose morphisms are continuous   homomorphisms between topologized semigroups. A bit wider is the category whose morphisms are partial homomorphisms, i.e., homomorphisms defined on subsemigroups. The largest is the category whose morphisms are multimorphisms between topologized semigroups.

By a {\em multimorphism}  between semigroups $X,Y$ we understand a multi-valued function $\Phi:X\multimap Y$ assigning to each point $x\in X$ a subset $\Phi(x)\subset Y$ so that $\Phi(x)\cdot \Phi(y)\subset \Phi(xy)$ for all $x,y\in X$.
For a multimorphism $\Phi:X\multimap Y$ the set $\Phi(X)=\bigcup_{x\in X}\Phi(x)$ is called the {\em range} of $\Phi$ and the set $\Phi^{-1}(Y)=\{x\in X:\Phi(x)\cap Y\ne\emptyset\}$ is the {\em domain} of $\Phi$.

A multimorphism $\Phi:X\multimap Y$ between semigroups $X,Y$ is called a {\em partial homomorphism} if for each $x\in X$ the set $\Phi(x)$ contains at most one point. Each partial homomorphism $\Phi:X\multimap Y$ can be identified with the unique function $\phi:\Phi^{-1}(Y)\to Y$ such that $\Phi(x)=\{\phi(x)\}$ for each $x\in\Phi^{-1}(Y)$. This function $\phi$ is a homomorphism from the subsemigroup $\Phi^{-1}(Y)$ of $X$ to the semigroup $Y$. A partial homomorphism $\Phi:X\multimap Y$ is called a {\em partial embedding} of $X$ to $Y$ if the function $\phi:\Phi^{-1}(Y)\to\Phi(X)$ is a topological isomorphism of the topologized semigroups $\Phi^{-1}(Y)$ and $\Phi(X)$.

For a class $\C$ of topologized semigroups by $\eC$, $\hC$, $\pC$, and $\peC$ we denote the categories whose objects are topologized semigroups in the class $\C$ and morphisms are isomorphic topological ${\mathsf e}$mbeddings, continuous $\mathsf h$omomorphisms, $\mathsf p$artial continuous homomorphisms and $\mathsf p$artial $\mathsf e$mbeddings of  topologized semigroups in the class $\C$, respectively.

In this paper we consider some concrete instances of the following general notion.

\begin{definition}
Let $\vec \C$ be a category of topologized semigroups and their multimorphisms. A topologized
semigroup $X$ is called {\em $\vec\C$-closed} if for any morphism $\Phi:X\multimap Y$ of the category $\vec \C$ the range $\Phi(X)$ is closed in $Y$.
\end{definition}

In particular, for a class $\C$ of topologized semigroup, a topologized semigroup $X$ is called
\begin{itemize}
\item {\em $\hC$-closed} if for any continuous homomorphism $f:X\to Y\in\C$ the image $f(X)$ is closed in $Y$;
\item {\em $\eC$-closed} if for each isomorphic topological embedding $f:X\to Y\in\C$ the image $f(X)$ is closed in $Y$;
\item {\em $\pC$-closed} if each closed subsemigroup of $X$ is $\hC$-closed;
\item {\em $\peC$-closed} if each closed subsemigroup of $X$ is $\eC$-closed.
\end{itemize}
The notations $\pC$ and $\peC$ are chosen because $\pC$-closed and $\peC$-closed topologized semigroups can be equivalently defined using $\mathsf p$artial homomorphisms and $\mathsf p$artial $\mathsf e$mbeddings of topologized semigroups.

The most natural candidate for the class $\C$ is the class
$\mathsf{T\!S}$ Hausdorff topological semigroups or its subclass $\mathsf{T\!G}$ consisting of topological groups.


In fact, under different names, $\vec\C$-closed topological semigroups have been already considered in mathematical literature.
In particular, ${\mathsf e}{:}\!\mathsf{T\!G}$-closed topological groups appeared in Raikov's characterization \cite{Raikov1946} of complete topological groups. In 1969 Stepp \cite{Stepp1969,Stepp75} introduced $\eTS$-closed and $\hTS$-closed topological semigroups calling them maximal and absolutely maximal semigroups, respectively. The study ${\mathsf h}{:}\!\mathsf{T\!G}$-closed and ${\mathsf p}{:}\!\mathsf{T\!G}$-closed topological groups (called $h$-complete and hereditarily $h$-complete topological groups, respectiely) was initiated by Dikranjan and Tonolo \cite{DTon} and continued by Dikranjan, Uspenskij \cite{DU}, Luka\'sc  \cite{Luk} and Banakh \cite{Ban}.  In \cite{Bardyla-Gutik-2012,
BGR,ChuchmanGutik2007,
GutikRepovs2008} Hausdorff $\eTS$-closed (resp. $\hTS$-closed) topological semigroups are called (absolutely) $H$-closed.

In this paper we are interested in detecting topological
semilattices which are $\vec\C$-closed for various categories $\vec\C$ of topologized semilattices. This topics will be continued in the paper \cite{BBc}.


By a {\em topologized semilattice} we understand a topological
space $X$ endowed with a binary operation $X\times X\to
X$, $(x,y)\mapsto xy$, which is associative, commutative and
idempotent (in the sense that $xx=x$ for all $x\in X$).
If the binary operation $X\times X\to X$ is (separately) continuous, then $X$ is called a ({\em semi}){\em topological semilattice}. 

Each semilattice carries a natural partial order $\le$ defined by $x\le
y$ iff $xy=x=yx$.
For an element $x$ of a semilattice $X$ by ${\downarrow}x=\{y\in X:y\le x\}$ and ${\uparrow}x=\{y\in X:x\le y\}$ we denote the {\em lower} and {\em upper cones} of $x$ in $X$. For a subset $A\subset X$ we put ${\downarrow}A:=\bigcup_{a\in A}{\downarrow}a$ and
${\uparrow}A:=\bigcup_{a\in A}{\uparrow}a$ be the {\em lower} and {\em upper sets} of $A$ in the semilattice $X$.

A subset $C$ of a semilattice $X$ is called a {\em
chain} if $xy=yx\in\{x,y\}$ for any $x,y\in C$. This is equivalent to
saying that any two elements $x,y\in C$ are comparable with respect
to the partial order $\le$. A semilattice is called {\em linear} if it is a chain.

A semilattice $X$ is called {\em
chain-finite} if each chain in $X$ is finite. A topological semilattice $X$ is called {\em chain-compact} if each closed chain in $X$ is compact. It is clear
that a (discrete) topological semilattice is chain-compact if (and only if) it is chain-finite.
In Theorem~\ref{t:main3} we shall prove that a Hausdorff topological semilattice is chain-compact if and only if it is {\em $k$-complete} in the sense that each non-empty chain $C\subset X$ has the
greatest lower bound $\inf C$ and the least upper bound $\sup C$,
which belong to the closure $\bar C$ of $C$ in $X$.

By $\TsL$ (resp. $\sTsL$) we denote the class of Hausdorff (semi)topological semilattices and by $\TsLone$ (resp. $\sTsLone$) the class of (semi)topological semilattices satisfying the separation axiom $T_1$ (= all finite subsets are closed). The class $\TsL$ is contained in the class $\TS$ of Hausdorff topological semigroups.


Next we introduce a (relatively small) class $\Ord$ of ordinal topological semilattices. Each ordinal $\alpha$ is identified with the set $\{\beta:\beta<\alpha\}$ of smaller ordinals and is endowed with the interval topology generated by the subbase consisting of the intervals $(\leftarrow,\beta):=\{x\in\alpha:x<\beta\}$ and $(\beta,\to):=\{x\in\alpha:x>\beta\}$ where $\beta\in\alpha$. Each ordinal is locally compact, zero-dimensional, and hereditarily normal, see \cite[2.7.5, 3.12.3]{Engelking1989}. Moreover, an ordinal is compact if and only if it is a successor ordinal. So, each limit ordinal $\lambda$ is a non-closed subspace of the ordinal $\lambda+1$.

Any ordinal  endowed with the operation of minimum or maximum is a topological semilattice, called an {\em ordinal semilattice}. The class of ordinal semilattices will be denoted by $\Ord$. It follows from \cite[3.12.3]{Engelking1989} that each ordinal semilattice is hereditarily normal.


A Hausdorff topological semilattice $X$ is called a {\em local $\w_{\max}$-semilattice} (resp. {\em local $\w_{\min}$-semilattice}) if each non-isolated point $x\in X$ is contained in an closed-and-open chain $C_x\subset X$, which is algebraically isomorphic to the ordinal $\w+1=[0,\w]$ endowed with the semilattice operation of maximum (resp. minimum). By $\TsLw$ we denote the class consisting of all local $\w_{\max}$-semilattices and all local $\w_{\min}$-semilattices. Each topological semilattice in the class $\TsLw$ is zero-dimensional, Hausdorff and locally compact. On the other hand, this class contains separable topological semilattices, which are not normal, see Example~\ref{ex:6.3}.

It is easy to see that $$\TsLw\cup\Ord\subset\TsL\subset \TsLone\subset \sTsLone.$$

In \cite{Stepp75} Stepp proved that a discrete topological
semilattice is chain-finite if and only if it is ${\mathsf h}{:}\mathsf{\!T\!sL}$-closed. In this paper we shall extend this characterization of Stepp finding many other closedness conditions,  equivalent to chain-finiteness (or chain-compactness) of topological semilattices. Among such conditions an important role is due to the $\vec\C$-closedness in the categories whose morphisms are $T_i$-multimorphisms bewteen topologized semilattices for $i\in\{1,2\}$.

To define $T_i$-multimorphisms between topologized semilattices we need to recall some information on $\theta$-closed sets in topological
spaces.

A subset $A$ of a topological space $X$ is called {\em
$\theta$-closed} in $X$ if each point $x\in X\setminus A$ has a
{\em closed} neighborhood $\bar O_x\subset X$ that does not
intersect $A$. It follows that each $\theta$-closed set is closed.
A subset $A$ of a regular topological space $X$ is closed if and
only if $A$ is $\theta$-closed in $X$. Moreover, a topological
space $X$ is regular if and only if each closed subset of $X$ is
$\theta$-closed. A topological space $X$ is Hausdorff if and only
if each singleton $\{x\}\subset X$ is $\theta$-closed in $X$.
$\theta$-Closed sets were introduced by Velichko \cite{Velichko} in 1966 and is an important concept in the theory of $H$-closed topological spaces \cite{Vel2}. Observe that a topological space $X$ satisfies the separation axiom $T_1$ (resp. $T_2$) if and only if each singleton $\{x\}\subset X$ is closed (resp. $\theta$-closed) in $X$. For a uniform presentation of our
results it will be convenient to call $\theta$-closed sets {\em
$T_2$-closed} and closed sets {\em $T_1$-closed}. 
Observe that a topological space $X$ satisfies the separation axiom $T_i$ for $i\in\{1,2\}$ if and only if each singleton $\{x\}\subset X$ is $T_i$-closed in $X$.

A multi-valued map $\Phi:X\multimap Y$ between topological spaces $X,Y$ is
called a
\begin{itemize}
\item {\em $T_i$-multimap} for $i\in\{1,2\}$ if for any $x\in
X$ the set $\Phi(x):=\{y\in X:(x,y)\in\Phi\}$ is $T_i$-closed in $Y$;
\item {\em upper semicontinuous} if for any closed set
$F\subset Y$ the preimage $\Phi^{-1}(F):=\{x\in X:\Phi(x)\cap F\ne\emptyset\}$ is closed in $X$.
\end{itemize}
A multimorphism $\Phi:X\multimap Y$ between topologized semigroups
is called a {\em $T_i$-multimorphism} for $i\in\{1,2\}$ if it is a $T_i$-multimap.

We recall that a multimap $\Phi:X\multimap Y$ between semigroups is called a {\em multimorphism} if $\Phi(x)\cdot \Phi(x')\subset \Phi(xx')$ for any $x,x'\in X$. Here $\Phi(x)\cdot\Phi(x')=\{yy':y\in\Phi(x),\;y'\in
\Phi(x')\}$ is the product of the sets $\Phi(x)$ and $\Phi(x')$ in
the semigroup $Y$.

Observe that each continuous homomorphism $h:X\to Y$ between
topologized semigroups is an upper semicontinuous multimorphism.
Moreover, if the space $Y$ is a $T_i$-space for $i\in\{1,2\}$, then $h$ is a $T_i$-multimorphism.

For $i\in\{1,2\}$ and a class $\C$ of topologized semigroups by $\mathsf{m}_{\mathsf{i}}{:}\C$ we denote the category whose objects are topologized semigroups in the class $\C$ and morphisms are upper semicontinuous $T_i$-multimorphisms between topologized semigroups in the class $\C$. If the class $\C$ consists of Hausdorff topologized semigroups, then $\mathsf{m}_{\!\mathsf{1}\!}{:}\C\supset \mathsf{m}_{\!\mathsf{2}}{:}\C\supset\pC\supset \hC\supset \eC$ and for every topologized semigroup we have the implications
$$\mbox{$\mathsf{m}_{\!\mathsf{1}\!}{:}\C$-closed $\Ra$
$\mathsf{m}_{\!\mathsf{2}}{:}\C$-closed $\Ra$
$\pC$-closed $\Ra$ $\hC$-closed $\Ra$ $\eC$-closed.}$$

The main results of the paper are drawn in Diagram 1 describing the relations between various
closedness properties of a Hausdorff topological
semilattice. In this diagram by $\C_1$ and $\C_2$ we denote any classes of topologized semilattices such that
$$\TsLone\subset\C_1\subset \sTsLone\mbox{ \ and \ }\Ord\cup\TsLw\subset \C_2\subset\TsL.$$
Curved arrows with inscription indicate the implications that hold under an additional assumption of the linearity or the discreteness of the semilattice.
{\small
$$\xymatrixcolsep{15pt}
\xymatrix{
\mbox{chain-finite}\ar@{<=>}[r]\ar@{=>}[d]&
\mbox{${\mathsf m}_{\!\mathsf 1\!}{:}\C_1$-closed}\ar@{<=>}[r]\ar@{=>}[dd]&
\mbox{$\mathsf{pe}{:}\C_1$-closed}\ar@{<=>}[r]\ar@{=>}[d]&
\mbox{$\mathsf{p}{:}\C_1$-closed}\ar@{=>}[r]\ar@{=>}[d]&
\mbox{$\mathsf h{:}\C_1$-closed}\ar@{=>}[r]\ar@{=>}[d]&
\mbox{$\mathsf e{:}\C_1$-closed}\ar@{=>}[d]
\ar@/_20pt/[ll]_{\mbox{\tiny{\it linear or discrete}}}
\\
\mbox{chain-compact}&
&\mbox{$\mathsf{pe}{:}\!\TS$-closed}\ar@{<=>}[r]\ar@{<=>}[d]
&\mbox{$\mathsf{p}{:}\!\TS$-closed}\ar@{=>}[r]\ar@{<=>}[d]&
\mbox{${\mathsf h}{:}\!\TS$-closed}\ar@{=>}[r]\ar@{=>}[d]&
\mbox{${\mathsf e}{:}\TS$-closed}\ar@{=>}[d] \ar@/^12pt/[l]^{\mbox{\tiny{\it linear}}}
\\
\mbox{$k$-complete}\ar@{<=>}[r]\ar@{<=>}[u]
&\mbox{${\mathsf m}_{\!\,\mathsf 2\!\,}{:}\C_2$-closed}\ar@{<=>}[r]&
\mbox{$\mathsf{pe}{:}\C_2$-closed}\ar@{<=>}[r]&
\mbox{$\mathsf{p}{:}\C_2$-closed}\ar@{=>}[r]&
\mbox{$\hC_2$-closed}\ar@{=>}[r]&
\mbox{$\eC_2$-closed.}\ar@{=>}[d]\\
&&\mbox{$\mathsf{pe}{:}\!\Ord$-closed}\ar@{<=>}[u]
&&&
\mbox{$\mathsf e{:}\!\TsLw$-closed.}\ar@/_35pt/[uuu]
|-{\hskip-5pt\mbox{\tiny{\it discrete}}}
}
$$
}
\centerline{Diagram 1}
\smallskip

The equivalence of the $\mathsf{e}{:}\!\TS$-closedness and $\mathsf{h}{:}\!\TS$-closedness for linear Hausdorff topological semilattices (in the second row of Diagram 1) was proved by Gutik and Repov\v s in \cite{GutikRepovs2008}.

\section{Closedness of $k$-complete topologized semilattices}\label{s:main}

A topologized semilattice $X$ is defined to be {\em $k$-complete} if each non-empty chain $C\subset X$ has the
greatest lower bound $\inf C$ and the least upper bound $\sup C$,
which belong to the closure $\bar C$ of $C$ in $X$. More information on $k$-complete topological semilattices can be found in \cite{BBc}.

A (topological) semigroup $S$ is called a ({\em topological}) {\em band} if each element $x\in S$ is an idempotent (i.e., $xx=x$).

\begin{theorem}\label{t:main1} Each upper semicontinuous $T_2$-multimorphism $\Phi:X\multimap Y$ from a  $k$-complete topo\-lo\-gized semilattice to a topological band $Y$ has $T_2$-closed image $\Phi(X)$ in $Y$.
\end{theorem}

\begin{proof} Assuming
that $\Phi(X)$ is not $T_2$-closed  in $Y$, we could find a point
$y\in Y\setminus \Phi(X)$ such that each closed neighborhood of $y$
in $Y$ meets the set $\Phi(X)$.

\begin{lemma}\label{l:w} For any decreasing sequence
$(U_n)_{n\in\w}$ of closed neighborhoods of $y$ with
$U_{n+1}U_{n+1}\subset U_n$, the set
$\bigcap_{n\in\w}\Phi^{-1}(U_n)$ is a non-empty closed
subsemilattice of $X$.
\end{lemma}

\begin{proof} For every $n\in\w$ the intersection $U_n\cap \Phi(X)$
is not empty, which implies that the preimage
$A_n=\Phi^{-1}(U_n):=\{x\in X:\Phi(x)\cap U_n\ne\emptyset\}$ is not
empty and closed in $X$ (because of the upper semicontinuity of $\Phi$). Next we show that $U_{n+1}U_{n+1}\subset U_n$
implies $A_{n+1}A_{n+1}\subset A_n$. Take any points $x,x'\in
A_{n+1}$ and find points $y\in\Phi(x)\cap U_{n+1}$ and
$y'\in\Phi(x')\cap U_{n+1}$. Taking into account that $\Phi$ is a
multimorphism, we conclude that $yy'\in \Phi(x)\cdot
\Phi(x')\subset\Phi(xx')$. Since $yy'\in U_{n+1}U_{n+1}\subset
U_n$, the point $xx'$ belongs to $\Phi^{-1}(yy')\subset
\Phi^{-1}(U_n)=A_n$. So, $A_{n+1}A_{n+1}\subset A_n$ for all
$n\in\w$. This inclusions imply that the intersection
$L=\bigcap_{n\in\w}A_n$ is a closed subsemilattice of $X$.

It remains to prove that this subsemilattice $L$ is not empty.
For every $n\in\w$ choose a point
$x_n\in A_{n+1}$. For any numbers $n<m$ consider the product
$x_n\cdots x_m\in X$. Using the inclusions $A_kA_k\subset A_{k-1}$
for $k\in\{m,m-1,\dots,n+1\}$, we can show that $x_n\cdots x_m\in
A_n$. Observe that for every $n\in\w$ the sequence $(x_{n}\cdots
x_m)_{m>n}$ is decreasing and by the $k$-completeness of $X$,
this chain has the greatest lower bound $\bar
x_n:=\inf_{m>n}x_n\cdots x_m$ in $X$, which belongs to
$\bar A_n=A_n$.

Taking into account that $x_n\cdots x_m\le x_k\cdots x_m$ for any
$n\le k<m$, we conclude that $\bar x_n\le \bar x_k$ for any $n\le
k$. By the $k$-completeness of $X$ the chain $\{\bar
x_n\}_{n\in\w}$ has the least upper bound $\bar x=\sup\{\bar
x_n\}_{n\in\w}$ in $X$. Since for every $k\in\w$ the chain $\{\bar
x_n\}_{n>k}$ is contained in $A_k$, the least upper
bound $\bar x=\sup\{\bar x_n\}_{n\in\w}=\sup\{\bar x_n\}_{n\ge k}$
belongs to $\bar A_k=A_k$ for all $k\in\w$. Then
$\bar x\in\bigcap_{n\in\w}A_n=L$, so
$L\ne\emptyset$.
\end{proof}
 Now we return back to the proof of Theorem~\ref{t:main1}. By
transfinite
induction, for every non-zero cardinal $\kappa$ we shall prove the
following statement:
\begin{itemize}
\item[$(*_\kappa)$] {\em for every family
$\{U_{\alpha}\}_{\alpha\in\kappa}$ of
neighborhoods of $y$ in $Y$, the set
$\bigcap_{\alpha\in\kappa}\Phi^{-1}(U_{\alpha})$ is not
empty.}
\end{itemize}
\smallskip

To prove this statement for the smallest infinite cardinal
$\kappa=\w$, fix an arbitrary sequence
$(U_{n})_{n\in\w}$ of closed
neighborhoods of $y$. Using the continuity of the semigroup
operation at $y=yy$, construct a decreasing sequence
$(W_n)_{n\in\w}$ of closed neighborhoods of $y$ such that
$W_n\subset \bigcap_{k\le n}U_{k}$ and
$W_{n+1}W_{n+1}\subset W_n$ for all $n\in\w$. By Lemma~\ref{l:w},
the intersection $\bigcap_{n\in\w}\Phi^{-1}(W_n)$ is not empty and
so is the (larger) intersection $\bigcap_{n\in\w}\Phi^{-1}(U_n)$.

Now assume that for some infinite cardinal $\kappa$ and all
cardinals $\lambda<\kappa$ the statement $(*_\lambda)$ has been
proved. To prove the statement $(*_\kappa)$, fix any family
$(U_{\alpha})_{\alpha\in\kappa}$ of closed
neighborhoods of $y$ in the topological semilattice $Y$.

For every $\alpha\in\kappa$, using the continuity of the
semilattice operation at $y=yy$ choose a decreasing sequence
$(U_{\alpha,n})_{n\in\w}$ of closed neighborhoods of $y$ such that
$U_{\alpha,0}\subset U_\alpha$ and $U_{\alpha,n+1}^2\subset
U_{\alpha,n}$ for all $n\in\w$. By Lemma~\ref{l:w}, the
intersection $L_\alpha=\bigcap_{n\in\w}\Phi^{-1}(U_{\alpha,n})$ is
a non-empty closed subsemilattice in $X$.

Then for every $\beta\in\kappa$ the intersection
$$L_{<\beta}=\bigcap_{\alpha<\beta}L_\alpha=
\bigcap_{\alpha<\beta}\bigcap_{n\in\w}\Phi^{-1}(U_{\alpha,n})$$ is
a closed subsemilattice of $X$. By the inductive assumption
$(*_{|\beta\times\w|})$, the semilattice $L_{<\beta}$ is not
empty.

Choose any maximal chain $M$ in $L_{<\beta}$. The
$k$-completeness of $X$ guarantees that $M$ has $\inf M\in \bar
M\subset L_{<\beta}$. We claim that
$x_\alpha:=\inf M$ is the smallest element of the semilattice
$L_{<\alpha}$. In the opposite case, we could find an element
$z\in
L_{<\alpha}$ such that $x_\alpha\not\le z$ and hence $x_\alpha
z<x_\alpha$. Then $\{x_\alpha z\}\cup M$ is a chain in
$L_{<\alpha}$ that properly contains the maximal chain $M$, which
is not possible. This contradiction shows that $x_\alpha=\inf M$
is
the smallest element of the semilattice $L_{<\alpha}$.

Observe that for any ordinals $\alpha<\beta<\kappa$ the inclusion
$L_{<\beta}\subset L_{<\alpha}$ implies $x_\alpha\le x_\beta$. So,
$\{x_\beta\}_{\beta\in\kappa}$ is a chain in $X$ and by the
$k$-completeness of $X$, it has
$\sup\{x_\beta\}_{\beta\in\kappa}\in
\bigcap_{\beta<\kappa}L_{<\beta}=\bigcap_{\alpha\in\kappa}
\bigcap_{n\in\w}\Phi^{-1}(U_{\alpha,n})\subset
\bigcap_{\alpha\in\kappa}\Phi^{-1}(U_\alpha)$. So, the latter set
is not empty and the
statement $(*_\kappa)$ is proved.
\smallskip

Since $\Phi$ is a $T_2$-multimap, for every $x\in X$ the set
$\Phi(x)$ is $T_2$-closed in $Y$. Since $y\notin \Phi(x)$, there
exists a closed neighborhood $U_x\subset Y$ of $y$ such that
$U_x\cap\Phi(x)=\emptyset$. Then the set $\bigcap_{x\in
X}\Phi^{-1}(U_{x})$ is empty, which contradicts the property
$(*_\kappa)$ for $\kappa=|X|$. This contradiction completes the
proof of the $T_2$-closedness of $\Phi(X)$ in $Y$.
\end{proof}

We recall that $\TsLone$ denotes the class of topological semilattices satisfying the separation axiom $T_1$ and $\mathsf{T\!S}$ is the class of Hausdorff topological semigroups. 
Theorem~\ref{t:main1} imply the following corollary.

\begin{corollary}\label{c:main1} Each $k$-complete topologized semilattice $X$ is $\mathsf{m}_{\!\,\mathsf{2}\!}{:}\!\TsLone$-closed and $\mathsf{p}{:}\!\mathsf{T\!S}$-closed.
\end{corollary}

\begin{proof} Theorem~\ref{t:main1} and  the definition of a $\mathsf{m}_{\!\,\mathsf{2}\!}{:}\C$-closed topologized semigroup imply that $X$ is $\mathsf{m}_{\!\,\mathsf{2}\!}{:}\!\TsLone$-closed.

To prove that $X$ is $\mathsf{p}{:}\!\mathsf{T\!S}$-closed, take any closed subsemilattice $Z\subset X$ and any continuous homomorphism $h:Z\to Y$ to a Hausdorff topological semigroup $Y$.
It follows that $h(Z)$ is a subsemilattice of $Y$ and so is its closure $\overline{h(Z)}$ in $Y$. Replacing $Y$ by $\overline{h(Z)}$, we can assume that $Y=\overline{h(Z)}$ is a topological semilattice.

 It follows that the multimap $\Phi:X\multimap Y$ defined by
$$
\Phi(x)=\begin{cases}
\{h(x)\},&\mbox{$x\in Z$},\\
\emptyset,&\mbox{$x\notin Z$},
\end{cases}
$$is a multimorphism. The continuity of $h$ guarantees that for any closed subset $F\subset Y$ the preimage $\Phi^{-1}(F)=h^{-1}(F)$ is closed in $Z$ and hence in $X$, too. This means that the multimap $\Phi$ is upper semicontinuous. Since the space $Y$ is Hausdorff, each singleton in $Y$ is $T_2$-closed, so $\Phi$ is a $T_2$-multimap. Since the topological semilattice $X$ is ${\mathsf m}_{\!\mathsf{2}}{:}\!\TsLone$-multiclosed, the image $\Phi(X)=h(Z)$ is closed in $Y=\overline{h(Z)}$.
\end{proof}

\section{Characterizing chain-compact topological semilattices}

In this section we shall apply Theorem~\ref{t:main1} to obtain a characterization of chain-compact Hausdorff topological semilattices. We recall that a topological semilattice $X$ is {\em chain-compact} if each closed chain in $X$ is compact. 

By $\Ord$ we denote the class of ordinal topological semilattices (i.e., ordinals endowed with the interval topology and the semilattice operation of minimum or maximum). Each cardinal $\kappa$ is identified with the smallest ordinal of cardinality $\kappa$. A cardinal $\kappa$ is {\em regular} if each subset $B\subset \kappa$ of cardinality $|B|<\kappa$ is upper bounded in $\kappa$.

\begin{theorem}\label{t:main3} Let $\C$ be a class of topological semigroups such that $\Ord\subset\C\subset\mathsf{T\!S}$. For a Hausdorff semitopological semilattice $X$ the following statements are equivalent:
\begin{enumerate}
\item $X$ is chain-compact;
\item each maximal chain in $X$ is compact;
\item $X$ is $k$-complete;
\smallskip
\item $X$ is $\mathsf{m}_{\mathsf{2}}{:}\!\TsLone$-closed;
\item $X$ is $\pC$-closed;
\item $X$ is $\peC$-closed;
\item any closed chain in $X$ is $\eC$-closed;
\item no closed chain in $X$ is topologically isomorphic to an infinite regular cardinal endowed with the semilattice operation of minimum or maximum.
\end{enumerate}
\end{theorem}

\begin{proof} First we prove the implications $(1)\Ra(2)\Ra(3)\Ra(1)$.
\smallskip

$(1)\Ra(2)$ Assume that $X$ is chain-compact and take a maximal chain $M\subset X$. The Hausdorff property of $X$ implies that the closure $\bar M$ of $M$ is a chain, equal to $M$ by the maximality of $M$. By the chain-compactness of $X$, the maximal chain $M=\bar M$ in $X$ is compact.
\smallskip

$(2)\Ra(3)$ Assume that each maximal chain in $X$ is compact. To prove that the topological semilattice $X$ is $\kappa$-complete, take any non-empty chain $C\subset X$. Using Zorn's Lemma, enlarge $C$ to a maximal chain $M\subset X$. By our assumption,  $M$ is a compact Hausdorff semilattice and so is its closed subsemilattice $\bar C$. By the compactness, the closed subchain $\overline{C}$ of $M$ has the smallest element $\inf\overline C$ (otherwise the sets $U_a=\{x\in\bar C:x>a\}$ would form an open cover $\{U_a:a\in\bar C\}$ of $\bar C$ without finite subcover). To see that $\inf \overline{C}=\inf C$, take any lower bound $b\in X$ for $C$ and observe that $C\subset{\uparrow}b$. Consequently,  $\inf\overline{C}\in\overline{C}\subset\overline{{\uparrow}b}={\uparrow}b$ and finally $b\le\inf \overline{C}$. So, $\inf\overline C$ is the largest lower bound for $C$ and hence $\inf C=\inf\overline C\in\overline C$. By analogy we can prove that $\sup C=\sup\overline C\in\overline C$. So, the semitopological semilattice $X$ is $k$-complete.
\smallskip

$(3)\Ra(1)$ Assume that the semitopological semilattice $X$ is
$k$-complete. We need to show that each closed chain $K\subset X$ is compact.
Denote by $\tau$ the interval topology on $K$ generated by the subbase consisting of the sets $K\setminus{\uparrow}x$ and $K\setminus {\downarrow}x$ where $x\in K$. The $k$-completeness of $X$ guarantees that each subset $C\subset K$ has $\inf C\in \overline C\subset \overline K=K$ and $\sup C\in\overline C\subset K$. By \cite[3.12.3]{Engelking1989}, the linearly ordered topological space $(K,\tau)$ is compact and Hausdorff.

It remains to prove that the topology $\tau$ coincides with the subspace topology on $K$, inherited from $X$. The closedness of upper and lower sets ${\uparrow}x$, ${\downarrow}x$ in $X$ implies that the identity map $K\to(K,\tau)$ is continuous. To prove that the identity map $(K,\tau)\to K$ is continuous, take any point $x\in K$ and any open neighborhood $U_x\subset K$ of $x$. We need to find an open set $V\in\tau$ such that $x\in V\subset U_x$.
For this, consider two closed subsets $F_{\uparrow}:=K\cap{\uparrow}x\setminus U_x$ and $F_{\downarrow}:=K\cap{\downarrow}x\setminus U_x$ of $K$.


Four cases are possible.

i) If both sets $F_{\downarrow}$ and $F_{\uparrow}$ are empty, then we put $V:=K$ and observe that $x\in V=K\subset U_x$.

ii) If both sets $F_{\downarrow}$ and $F_{\uparrow}$ are not empty, then by the $k$-completeness of $X$ there exist $a:=\sup F_{\downarrow}\in\overline F_{\downarrow}=F_{\downarrow}=K\cap{\downarrow}x\setminus U_x$ and
$b:=\inf F_{\uparrow}\in \overline F_{\uparrow}=F_{\uparrow}=K\cap{\uparrow}x\setminus U_x$. It follows that $a<x<b$ and hence the order interval $V=K\setminus({\downarrow}a\cup{\uparrow}b)$ is the required open set in $(K,\tau)$ such that
$$K\setminus U_x=(K\cap{\uparrow}x\setminus U_x)\cup(K\cap{\downarrow}x\setminus U_x)\subset (K\cap{\uparrow}b)\cup(K\cap{\downarrow}a)=K\setminus V\subset K\setminus\{x\}$$ and hence $x\in V\subset U_x$.

iii) If $F_{\downarrow}$ is empty and $F_{\uparrow}$ is not empty, then by the $k$-completeness of $X$, the chain $F_{\uparrow}$ has $\inf F_{\uparrow}\in\overline F_{\uparrow}=F_{\uparrow}$. Then the open set $V=K\setminus {\uparrow}\inf F_{\uparrow}\in\tau$ is a required open set with $x\in V\subset U_x$.

iv) If $F_{\downarrow}$ is not empty and $F_{\uparrow}$ is empty, then by the $k$-completeness of $X$, the chain $F_{\downarrow}$ has $\sup F_{\downarrow}\in\overline F_{\downarrow}=F_{\downarrow}$. Then the open set $V=K\setminus {\downarrow}\sup F_{\downarrow}\in\tau$ is a required open set with $x\in V\subset U_x$.

This completes the proof of the continuity of the identity map $(K,\tau)\to K$. Now the compactness of the topology $\tau$ implies the compactness of the subspace topology on the maximal chain $K$.
\medskip

Next, we prove the implications  $(3)\Ra(4)\Ra(5)\Ra(6)\Ra(7)\Ra(8)\Ra(3)$. The implications $(3)\Ra(4)\Ra(5)$ were proved in (the proof of) Corollary~\ref{c:main1} and $(5)\Ra(6)\Ra(7)$ are trivial.
\smallskip

$(7)\Ra(8)$ Assume that some closed chain $C\subset X$ is topologically isomorphic to an infinite regular cardinal $\kappa$ endowed with the operation of minimum or maximum. Being a limit ordinal, the cardinal $\kappa=[0,\kappa)$ is a non-closed subset of the ordinal $\kappa+1=[0,\kappa]\in\Ord\subset\C$. Fix any topological isomorphism $h:C\to\kappa\subset\kappa+1$ and observe that its image $h(C)$ is not closed in $\kappa+1$, witnessing that the topological semilattice $C$ is not $\eC$-closed.
\smallskip

$(8)\Ra(3)$ Assuming that the topological semilattice $X$ is not $\kappa$-complete, we can find a chain $C\subset X$ such that $\inf C\notin\bar C$ or $\sup C\notin\bar C$. We can assume that $C$ is of the smallest possible cardinality, which means that each chain $K\subset X$ of cardinality $|K|<|C|$ has $\inf K\in\bar K$ and $\sup K\in \bar K$. It is clear that the cardinal $\kappa:=|C|$ is infinite.

First we consider the case of $\sup C\notin\bar C$. Let $\{z_\alpha\}_{\alpha<\kappa}$ be an enumeration of the set $C$ by ordinals $\alpha<\kappa$ where $\kappa=|C|$.

By transfinite induction we shall construct an increasing transfinite sequence $(x_\alpha)_{\alpha<\kappa}$ of points of $\bar C$ such that
\begin{enumerate}
\item[(i)] $x_{\alpha+1}\in C$ and $x_{\alpha+1}>\max\{x_\alpha,z_\alpha\}$ for any $\alpha<\kappa$;
\item[(ii)] $x_\alpha=\sup\{x_\beta:\beta<\alpha\}$ for any limit ordinal $\alpha<\kappa$.
\end{enumerate}
We start the inductive construction letting $x_0=z_0$. Assume that for some ordinal $\alpha<\kappa$ an increasing sequence $\{x_\beta\}_{\beta<\alpha}\subset \bar C$ has been constructed. If $\alpha$ is a limit ordinal, then put $x_\alpha:=\sup\{x_\beta\}_{\beta<\alpha}$ (the supremum exists by the minimality of the cardinal $\kappa=|C|>|\alpha|$). Next, assume that $\alpha$ is not limit and hence $\alpha=\beta+1$ for some ordinal $\beta$. Consider the element $z=\max\{z_\beta,x_\beta\}\in\bar C$ and observe that $z$ is not a maximal element of $\bar C$ (otherwise $\sup C=\max\bar C=z\in\bar C$, which contradicts the choice of $C$). Then the open subset $U=\{x\in\bar C:x>z\}$ is not empty and by the density of $C$ in $\bar C$, there exists a point $x_\alpha\in U\cap C$ satisfying the condition (i). This completes the inductive step.

After completing the inductive construction, consider the chain $L=\{x_\alpha\}_{\alpha<\kappa}$ in $X$. Observe that for every $\alpha<\kappa$ the subchain $L_\alpha:=\{x_\beta\}_{\beta\le\alpha}$ is closed-and-open in $L$ as $L_\alpha=L\cap{\downarrow}x_\alpha=L\setminus{\uparrow}x_{\alpha+1}$. The conditions (i),(ii) and the minimality of the cardinal $\kappa>|L_\alpha|$ guarantee that the chain $L_\alpha$ is $\kappa$-complete (as each chain $K\subset L$ has $\inf K\in\bar K\cap L$ and $\sup K\in\bar K\cap L$) and hence compact (by the equivalence $(1)\Leftrightarrow(3)$). Then the chain $L$ is locally compact.

We claim that the chain $L$ is closed in $X$. To derive a contradiction, assume that $\bar L\ne L$ and fix a point $x\in\bar L\setminus L$. First we show that $x\ge x_{\alpha+1}$ for every $\alpha<\kappa$. Since $\bar L$ is a chain, the assumption $x\not\ge x_{\alpha+1}$ implies $x<x_{\alpha+1}$ and hence $x\in \bar L\setminus{\uparrow}x_{\alpha+1}=\bar L_\alpha=L_\alpha\subset L$, which contradicts the choice of $x$. So, $x\ge x_{\alpha+1}>z_\alpha$ for any $\alpha<\kappa$ and hence $x$ is an upper bound for the chain $C=\{z_\alpha\}_{\alpha<\kappa}$. On the other hand, for any upper bound $b$ of $C$ in $X$, we get $C\subset{\downarrow}b$, $x\in\bar L\subset\bar C\subset \overline{{\downarrow}b}={\downarrow}b$ and $x\le b$. This means that $\sup C=x\in\bar L\subset\bar C$, which contradicts the choice of the chain $C$. This contradiction shows that the chain $L$ is closed in $X$.

Next, we show that the cardinal $\kappa$ is regular. In the opposite case, we can find a cofinal subset $\Lambda\subset \kappa$ of cardinality $|\Lambda|<\kappa$. By the minimality of the cardinal $\kappa$, the chain $\{x_\alpha\}_{\alpha\in\Lambda}$ has the smallest upper bound $u\in\overline{\{x_\alpha\}}_{\alpha\in\Lambda}\subset\bar L$. The cofinality of $\Lambda$ in $\kappa$ guarantees that $u=\sup L$. Since the chain $L$ does not have the largest element, $u\in \bar L\setminus L$, which contradicts the closedness of $L$ in $X$. This contradiction completes the proof of the regularity of the cardinal $\kappa$.

Endow the cardinal $\kappa=[0,\kappa)$ with the semilattice operation of minimum and consider the homomorphism $h:L\to\kappa$, $h:x_\alpha\mapsto\alpha$. Observe that for every ordinal $\alpha<\kappa$ the subbasic sets $(\leftarrow,\alpha),(\alpha,\to)\subset [0,\kappa)$ have open preimages $h^{-1}\big((\leftarrow,\alpha)\big)=L\setminus{\uparrow}\alpha$ and $h^{-1}\big((\alpha,\to)\big)=L\setminus {\downarrow}\alpha$ in $L$, which implies that the homomorphism $L$ is continuous. The compactness and openness of the sets $L_\alpha=L\cap{\downarrow}\alpha$ in $L$ imply that the continuous map $h$ is closed, so is a topological isomorphism. Now we see that the topological semilattice $X$ contains a closed subsemilattice $L$, which is topologically isomorphic to the infinite regular cardinal $\kappa$ endowed with the semilattice operation of minimum.

By analogy, we can prove that the assumption $\inf C\notin\bar C$ implies the existence of a closed chain in $X$, which is topologically isomorphic to the regular cardinal $\kappa$ endowed with the semilattice operation of maximum.
\end{proof}

\section{Closedness of chain-finite semilattices}

The main result of this section is the following theorem.

\begin{theorem}\label{t:main2} Each $T_1$-multimorphism $\Phi:X\multimap Y$ from a chain-finite topologized semilattice $X$ to a semitopological semilattice $Y$ has closed image $\Phi(X)$ in $Y$.
\end{theorem}

\begin{proof} Assuming that $\Phi(X)$ is not closed in $Y$, fix a point $y\in \overline{\Phi(X)}\setminus\Phi(X)$.

Let $\tau_y$ be the family of all open neighborhoods of $y$ in $Y$.
In the semilattice $X$ consider the subset $$X_y=\{x\in X:\exists
W\in\tau_y\;\;W\cap\Phi(X)\subset\Phi({\uparrow}x)\}.$$ The set $X_y$ contains the smallest element of the semilattice $X$ and hence is
not empty. Let $e$ be a maximal element of $X_y$ and $W\in \Tau_y$
be an open neighborhood of $y$ such that $W\cap\Phi(X)\subset\Phi({\uparrow}e)$. Since the set $\Phi(e)\not\ni y$ is closed in $Y$, we can replace $W$ by a smaller neighborhood and additionally assume that $W\cap\Phi(e)=\emptyset$.

 By induction we shall construct two sequences of points $(y_n)_{n\in\w}$ and $(w_n)_{n\in\w}$ in $W\cap\Phi(X)$ such that for every $n\in\w$ the following conditions are satisfied:
\begin{itemize}
\item[(i)] $y_n,y_ny\in W$;
\item[(ii)] $w_{n+1}\notin\Phi({\uparrow}x_n)$ where $x_n$ is the smallest element of the semilattice $S_n=\{x\in{\uparrow}e:y_n\in\Phi(x)\}$;
\item[(iii)] $y_{n+1}=y_nw_{n+1}$.
\end{itemize}

To start the inductive construction, choose any point $y_0$ in the set $\Phi(X)\cap W\cap\{z\in Y:zy\in W\}$ and observe that the condition (i) is satisfied. Assume that for some $n\in\w$ a point  $y_n\in W$ with $y_ny\in W$ has been constructed.

Taking into account that $\Phi$ is a multimorphism and $y_n\in \Phi(X)\cap W\subset\Phi({\uparrow}e)$, we conclude that the set $S_n:=\{x\in {\uparrow}e:y_n\in\Phi(x)\}$ is a non-empty subsemilattice of the chain-finite semilattice $X$, so $S_n$ has the smallest element $x_n=\inf S_n$. It follows from $y_n\in W\subset Y\setminus\Phi(e)$ that $x_n\ne e$ and hence $x_n>e$. The separate continuity of the semigroup operation on $Y$ implies that the set $W_n=\{z\in W:zy_n,zy\in W\}$ is an open neighborhood of $y$. The maximality of $e$ guarantees that $W_n\cap\Phi(X)\not\subset\Phi({\uparrow}x_n)$. So, we can
choose a point $w_{n+1}\in W_n\cap\Phi(X)\setminus\Phi({\uparrow}x_n)$ and put $y_{n+1}=w_{n+1}y_n$. This completes the inductive step.

After completing the inductive construction, consider the sequence $(x_n)_{n\in\w}$ of the smallest elements of the semilattices $S_n=\{x\in{\uparrow}e:y_n\in \Phi(x)\}$. We claim that the sequence $(x_n)_{n\in\w}$ is strictly decreasing. Indeed, for every $n\in\w$, we can use the inclusion $w_{n+1}\in W\cap\Phi(X)\subset\Phi({\uparrow}e)$ to choose an element $x\in {\uparrow}e$ with $w_{n+1}\in\Phi(x)$.
The point $x$ does not belong to ${\uparrow}x_n$ as $w_{n+1}\notin\Phi({\uparrow}x_n)$. Then $y_{n+1}=w_{n+1}y_n\in \Phi(x)\cdot\Phi(x_n)\subset\Phi(xx_n)=\Phi(x_nx)$ and hence $x_nx\in S_{n+1}$ and $x_{n+1}\le x_nx<x_n$. On the other hand, the chain-finite semilattice $X$ does not contain strictly decreasing sequences.
\end{proof}

We recall that $\sTsLone$ denotes the class of semitopological semilattices satisfying the separation axiom $T_1$.

\begin{corollary}\label{c:main2} Each chain-finite topologized semilattice is $\mathsf{m}_{\!\mathsf{1}\!}{:}\sTsLone$-closed and $\mathsf{p}{:}\sTsLone$-closed.
\end{corollary}


\section{Characterizing chain-finite topological semilattices}

In this section we characterize chain-finite Hausdorff topological semilattices. We recall that $\TsLone$ (resp. $\sTsLone$) denotes the class of (semi)topological semilattices satisfying the separation axiom $T_1$.

\begin{theorem}\label{t:finite} For a Hausdorff (semi)topological semilattice $X$ and a class $\C\subset\sTsLone$ containing the class $\TsLone$ (resp. $\sTsLone$), the following conditions are equivalent:
\begin{enumerate}
\item $X$ is chain-finite;
\item $X$ is $\mathsf{m}_{\!\mathsf{1}\!}{:}\C$-closed;
\item $X$ is $\pC$-closed;
\item $X$ is $\peC$-closed;
\item each closed countable chain in $X$ is $\mathsf{e}{:}\C$-closed.
\end{enumerate}
\end{theorem}

\begin{proof} The implication $(1)\Ra(2)$ is proved in Corollary~\ref{c:main2} and $(2)\Ra(3)$ can be proved by analogy with Corollary~\ref{c:main1}. The implications $(3)\Ra(4)\Ra(5)$ are trivial. To prove the implication $(5)\Ra(1)$, assume that each closed countable chain in $X$ is $\eC$-closed. To derive a contradiction, assume that the semilattice $X$ contains an infinite chain.

Applying Ramsey Theorem, we can show that $X$ contains a sequence $C=\{x_n\}_{n\in\w}$ which is either strictly increasing or strictly decreasing. The Hausdorff property of $X$ and the separate continuity of the semilattice operation on $X$ imply that for every $x\in X$ the lower and upper cones ${\downarrow}x$ and ${\uparrow}x$ are closed in $X$.
This fact can be used to show that each point $x_n$ is isolated in the chain $C$ and hence $C$ is topologically isomorphic to the ordinal $\w$ endowed with the semilattice operation of maximum or minimum. Since the ordinal $\w$ is not $\Tau_1$-closed, the countable chain $C$ cannot be closed in $X$. Consequently, the chain $\bar C$ contains a point $p\notin C$.

If the sequence $C=\{x_n\}_{n\in\w}$ is strictly increasing then $p=\sup C$ and hence $\bar C\setminus C$ consists of a unique point $p$.
It can be shown that the closed chain $\bar C=C\cup\{p\}$ is a topological semilattice.
It follows that the upper cone ${\uparrow}p\cap\bar C$ of $p$ in $\bar C$ is not open in the countable chain $\bar C=C\cup\{p\}$.

If $C=\{x_n\}_{n\in\w}$ is strictly decreasing, then $p=\inf C$ is a non-isolated point of $\bar C$ with $\bar C={\uparrow}p\cup{\downarrow}p$ and $p\ne xy$ for any points $x,y\in\bar C\setminus\{p\}$. In both cases Lemma~\ref{l:non} (proved below) implies that the Hausdorff topological semilattice $\bar C$ is not $\mathsf{e}{:}\C$-closed.
\end{proof}

\begin{lemma}\label{l:non} Let $\C\in\{\TsLone,\sTsLone\}$. If a semitopological semilattice $X\in\C$ is  $\C$-closed, then
\begin{enumerate}
\item for any $p\in X$ the upper cone ${\uparrow}p$ is open in $X$;
\item any point $p\in X$ with $X={\uparrow}p\cup{\downarrow}p$ is either isolated or $p=xy$ for some points $xy\in X\setminus\{p\}$.
\end{enumerate}
\end{lemma}

\begin{proof} 1. Assume that for some point $p\in X$ the upper cone ${\uparrow}p$ is not open in $X$. We claim that $p$ belongs to the closure of the set $X\setminus{\uparrow}p$. In the opposite case $p$ has an open neighborhood $O_p\subset X$ such that $O_p\subset {\uparrow}p$ and the upper cone ${\uparrow}p=\{x\in X:px\in O_p\}$ of $p$ is open in $X$ by the separate continuity of the semilattice operation on $X$.
But this contradicts our assumption.

Choose any point $p'\notin X$ and consider the space $Y=X\cup \{p'\}$ endowed with the $T_1$-topology consisting of sets $W\subset Y$ satisfying two conditions:
\begin{itemize}
\item $W\cap X$ is open in $X$ and
\item if $p'\in W$, then $V\setminus{\uparrow}p\subset W$ for some neighborhood $V$ of $p$ in $X$.
\end{itemize}

Extend the semilattice operation from $X$ to $Y$ letting $p'p'=p'$ and $p'x=xp'=xp$ for $x\in X$. It is easy to check that $Y$ is a semitopological semilattice in the class $\C$ containing $X$ as an open non-closed subsemilattice and witnessing that $X$ is not $\C$-closed.
\smallskip

2. Take a point $p\in X$ with $X={\uparrow}p\cup{\downarrow}p$ and assume that $p\ne xy$ for any points $x,y\in X\setminus\{p\}$.
The latter condition implies that the set $V={\uparrow}p\setminus\{p\}=X\setminus{\downarrow}p$ is a subsemilattice in $X$. By the first statement, the upper cone ${\uparrow}p$ is open, so $V$ is an open subsemilattice in $X$. Assuming that the point $p$ is not isolated, we conclude that $p$ belongs to the closure of $V$.

Consider the set $Y=(X\times\{0\})\cup(V\times\{1\})$ endowed with the $T_1$-topology $\tau_Y$ consisting of all sets $U\subset Y$ such that
\begin{itemize}
\item for every $x\in X\setminus {\uparrow}p$ with $(x,0)\in U$ there exists a neighborhood $U_x\subset X$ of $x$ such that $U_x\times\{0\}\subset U$;
\item for every $x\in X\setminus{\downarrow}p$ with $(x,1)\in U$ there exists a neighborhood $U_x\subset V$ of $x$ such that $U_x\times\{1\}\subset U$;
\item for every $x\in X\setminus{\downarrow}p$ with $(x,0)\in U$ there exist a neighborhood $U_x\subset V$ of $x$ and a neighborhood $U_p\subset X$ of $p$ such that $ (U_x\times\{0\})\cup\big((U_p\cap V)\times\{0,1\}\big)$;
\item if $(p,0)\in U$, then there exists a neighborhood $U_p\subset X$ of $p$ such that $(U_p\times\{0\})\cup\big((U_p\cap V)\times\{1\}\big)\subset U$.
\end{itemize}
On the topological space $Y=(X\times\{0\})\cup(V\times\{1\})$ consider the semilattice operation
defined by the formula $(x,i)\cdot (y,j)=(xy,\min\{i,j\})$.

It can be shown that $Y$ is a semitopological semilattice in the class $\C$ and the map $f:X\to Y$ defined by
$$f(x)=\begin{cases}(x,0)&\mbox{if $x\in{\downarrow}p$},\\
(x,1)&\mbox{otherwise}
\end{cases}
$$is an isomorphic topological embedding of $X$ with dense image $f(X)\ne Y$ in $Y$, witnessing that $X$ is not $\C$-closed.
\end{proof}

\begin{problem}\label{prob} Is each $\mathsf{e}{:}\!\TsLone$-closed Hausdorff topological semilattice $X$ chain-finite?
\end{problem}

The answer to Problem~\ref{prob} is affirmative for discrete or linear Hausdorff topological semilattices, see Theorems~\ref{t:discrete} and \ref{t:linear}.

\section{Characterizing discrete chain-finite semilattices}

The main results of this paper is the following theorem showing that for discrete topological semilattices many conditions of Diagram 1 are equivalent to the chain-finiteness.

\begin{theorem}\label{t:discrete} Let $\C$ be a class of topologized semilattices such that $\TsLw\subset\C\subset \sTsLone$. For a discrete topological semilattice $X$ the following statements are equivalent:
\begin{enumerate}
\item $X$ is chain-finite;
\item $X$ is chain-compact;
\item $X$ is $k$-complete;
\smallskip
\item $X$ is $\mathsf{m}_{\!\mathsf{1}\!}{:}\C$-closed;
\item $X$ is $\mathsf{m}_{\!\mathsf{2}\!}{:}\C$-closed;
\smallskip
\item $X$ is $\pC$-closed;
\item $X$ is $\hC$-closed;
\item $X$ is $\eC$-closed;
\item $X$ is ${\mathsf e}{:}\!\TsLw$-closed.
\end{enumerate}
\end{theorem}

\begin{proof} Taking into account Theorems~\ref{t:main2}, \ref{t:main3}, Theorem~\ref{t:finite} and the trivial implications in Diagram 1, it suffices to prove that each ${\mathsf e}{:}\!\TsLw$-closed discrete topological semilattice $X$ is chain-finite. To derive a contradiction, assume that $X$ contains an infinite chain $C$. Replacing $C$ by a smaller chain, we can assume that the chain $C$ is isomorphic to the ordinal $\w$ endowed with the semilattice operation of minimum or maximum.

Consider the Stone-\v Cech compactification $\beta X$ of the discrete space $X$. Elements of $\beta X$ are ultrafilters on $X$. Points of the discrete space $X$ can be identified with principal ultrafilters on $X$. For two ultrafilters $\U,\V\in\beta X$ their product $\U*\V$ is defined as the ultrafilter generated by the base consisting of the sets $\bigcup_{x\in U}xV_x$ where $U\in\U$ and $(V_x)_{x\in U}\in\V^U$. It is known \cite[\S 4.1]{HS} that the operation $*$ turns $\beta X$ into a semigroup, containing $X$ as a subsemigroup. Fix any free ultrafilter $\F$ containing the chain $C$. Taking into account that $FF\subset F$ for any subset $F\subset C$, we conclude that $\F*\F=\F$. So, $\F$ is an idempotent of the semigroup $\beta X$. The definition of the semigroup operation on $\beta X$ ensures that principal ultrafilters commute with all other elements of the semigroup $\beta X$. This implies that the subsemigroup $S$ of $\beta X$, generated by the set $X\cup\{\F\}$, is a semilattice, equal to $X\cup \{x\F:x\in X^1\}$ where $X^1=X\cup\{1\}$ and $1$ is the largest element of $X$ if $X$ has the largest element and $1$ is an external unit for $X$ if $X$ does not have the largest element.

Endow $S$ with the subspace topology inherited from $\beta X$. For each ultrafilter $\U\in S$ a neighborhood base at $\U$ consists of the sets $\langle U\rangle=\{\mathcal V\in S:U\in\mathcal V\}$ where $U\in\U$.
Being a subspace of the zero-dimensional compact Hausdorff space $\beta X$, the space $S$ is zero-dimensional and Tychonoff. Moreover, $X$  is a dense discrete subspace in $S$.

We claim that the binary operation $*:S\times S\to S$, $(\U,\V)\to \U*\V$ is (jointly) continuous at each pair $(\U,\V)\in S\times S$. If $\U$ and $\V$ are principal ultrafilters, then they are isolated points of $S\subset \beta X$ and the operation $*$ is continuous at $(\U,\V)$ by the trivial reason.

Next, assume that $\U$ and $\V$ are free ultrafilters. Then $\U=u\F$ and $\V=v\F$ for some $u,v\in X^1=X\cup\{1\}$. Given any neighborhood $O_{\U*\V}$ of $\U*\V=(u\F)*(v\F)=uv*(\F\F)=uv\F$, find a set $F\in\F$ such that $F\subset C$ and $\langle uvF\rangle\subset O_{\U*\V}$. Observe that $O_\U:=\langle uF\rangle$ and $O_\V:=\langle vF\rangle$ are neighborhoods of $\U$ and $\V$, respectively, such that for any ultrafilters $\U'\in O_\U$ and $\V'\in O_\V$ we get $\U'*\V'\in \langle uvF\rangle\subset O_{\U*\V}$. Indeed, for the sets $uF\in\U'$ and $vF\in\V'$, we get
$$\U*\V\ni\bigcup_{x\in uF}xvF\subset uFvF=uvFF=uvF.$$ Consequently, $uvF\in \U'*\V'$ and $\U'*\V'\in\langle uvF\rangle\subset O_{\U*\V}$.

By analogy we can consider the case when one of the ultrafilters $\U$ or $\V$ is principal.

Observe that each non-isolated point $s\in S$ is equal to $x\F$ for some $x\in X^1$ and $O_s=\{\U\in S:xC\in\U\}=xC\cup\{x\F\}$ is an open-and-closed neighborhood of $s$ in $S$, algebraically isomorphic to the ordinal $\w+1$ endowed with the semilattice operation of minimum or maximum. This means that $S\in\Tau_\w$ is a zero-dimensional Hausdorff topological semilattice, containing the discrete topological semilattice $X$ as a dense (non-closed) subsemilattice, and witnessing that $X$ is not ${\mathsf e}{:}\!\TsLw$-closed.
\end{proof}

The semilattice $S$ in the proof of Theorem~\ref{t:discrete} needs not be normal.

\begin{example}\label{ex:6.3} There exists a discrete topological semilattice $X$ containing a countable infinite chain $C$ such that
\begin{enumerate}
\item for any free ultrafilter $\U\in\beta(X)$ containing $C$, the subsemilattice $X\cup\{\U\}\cup\{x\U:x\in X\}$ of $\beta(X)$ generated by $X\cup\{\U\}$ is not normal;
\item for any chain $D\subset X\setminus C$ and any free ultrafilter $\V\in\beta(X)$ containing $D$, the subsemilattice $X\cup\{\V\}\cup\{x\V:x\in X\}$ of $\beta(X)$ generated by $X\cup\{\V\}$ coincides with $X\cup\{\V\}$ and hence is a perfectly normal space with a unique non-isolated point $\V$.
\end{enumerate}
\end{example}

\begin{proof} Let $2^{\le \w}=2^\w\cup 2^{<\w}$ where $2^{<\w}=\bigcup_{n\in\w}2^n$. For two functions $x,y\in 2^{\le \w}$ let $xy\in 2^{\le\w}$ be a function defined on the set  $\dom(xy):=\{n\in\dom(x)\cap\dom(y):\forall k\le n\; x(k)=y(k)\}$ by the formula $xy(k)=x(k)=y(k)$ for $k\in\dom(xy)$. It is easy to see that the operation $2^{\le \w}\times 2^{\le\w}\to 2^{\le\w}$, $(x,y)\mapsto xy$, turns $2^{\le \w}$ into a semilattice.

Let $X=\w\cup 2^{\le\w}$, endow $X$ with the discrete topology, and extend the semilattice operation of $2^{\le\w}$ to $X$ letting $xn=nx:=x|n$ and $nm=\min\{n,m\}$ for any $x\in 2^{\le \w}$ and $n,m\in\w$.

We claim that the semilattice $X$ and the chain $C:=\w\subset X$ have  the required properties.
\smallskip

1. Take any free ultrafilter $\U$ on $X$ with $\w\in\U$ and consider the subsemilattice $S:=X\cup\{\U\}\cup\{x\U:x\in X\}$ generated by the set $X\cup\{\U\}$. It is easy to see that the set $Z=2^{<\w}\cup\{x\U:x\in 2^{\le \w}\}$ is closed-and-open in $S$, the countable set $2^{<\w}$ is dense in $Z$ and the set $\{x\U:x\in 2^\w\}$ of cardinality continuum is closed and discrete in $Z$. These two properties imply that the closed subspace $Z$ of $S$ is not normal (see \cite[1.5.10]{Engelking1989}) and then the space $S$ is not normal, too.
\smallskip

2. Take any chain $D\subset X\setminus\w$ and any free ultrafilter $\V\in\beta(X)$ containing $D$. It can be shown that for any $x\in X$ the product $x\V$ belongs to $2^{<\w}\cup\{\V\}$, which implies that $S:=X\cup\{\V\}$ is a subsemilattice of $\beta(X)$ with a unique non-isolated point.
Being a Tychonoff space with a unique non-isolated point, the space $S$ is zero-dimensional and hereditarily normal.  To see that $S$ is perfectly normal, we need to check that each closed subset $F\subset S$ is of type $G_\delta$. If $\V\notin F$, then the set $F$ is open in $S$ and hence of type $G_\delta$. So, we assume that $\V\in F$. It is easy to see that the chain $D$ is countable and hence can be written as the union $D=\bigcup_{n\in\w}D_n$ of an increasing sequence $(D_n)_{n\in\w}$ of finite sets. For every $n\in\w$ consider the open set  $F_n=F\cup(D\setminus D_n)$ in $S$ and observe that $F=\bigcap_{n\in\w}F_n$, witnessing that $F$ is a $G_\delta$-set in $X$.
\end{proof}

\section{Linear topological semilattices}

A topological semilattice $X$ is called {\em linear} if $X$ is a chain, i.e., $xy\in\{x,y\}$ for all $x,y\in X$. The following characterization of $\eTS$-closed and $\hTS$-closed linear topological semilattices was
proved by Gutik and Repov\v s in \cite{GutikRepovs2008}.

\begin{theorem}[Gutik, Repov\v s]\label{t:GR} For a linear Hausdorff topological semilattice $X$ the following conditions are equivalent:
\begin{enumerate}
\item $X$ is $\hTS$-closed;
\item $X$ is $\eTS$-closed;
\item each non-empty subset $C\subset X$ has $\inf C\in\overline{{\uparrow}C}$ and $\sup C\in\overline{{\downarrow}C}$.
\end{enumerate}
\end{theorem}

\begin{remark} The equivalence of the conditions (1) and (2) in Theorem~\ref{t:GR} does not hold for non-linear topological semilattices: by \cite{Bardyla-Gutik-2012}, there exists an $\eTS$-closed Hausdorff topological semilattice which is not  $\hTS$-closed.
\end{remark}

The main result of this section is the following characterization of finite linear topological semilattices.

\begin{theorem}\label{t:linear} For a Hausdorff linear (semi)topological semilattice $X$ and a class $\C\subset\sTsLone$ containing the class $\TsLone$ (resp. $\sTsLone$),  the following conditions are equivalent:
\begin{enumerate}
\item $X$ is finite;
\item $X$ is $\mathsf{m}_{\!\mathsf{1}\!}{:}\C$-closed;
\item $X$ is $\pC$-closed;
\item $X$ is $\hC$-closed;
\item $X$ is $\eC$-closed.
\end{enumerate}
\end{theorem}

\begin{proof} The equivalence of the conditions $(1)$--$(3)$ follows from Theorem~\ref{t:finite}. The implications $(3)\Ra(4)\Ra(5)$ are trivial. To prove that $(5)\Ra(1)$, assume that the semilattice $X$ is $\eC$-closed. By Lemma~\ref{l:non}(2) each point $p\in X$ is isolated and hence the space $X$ is discrete. Applying Theorem~\ref{t:discrete}, we conclude that the discrete $\eC$-closed semilattice $X$ is chain-finite and being linear, is finite.
\end{proof}

\begin{remark}
Theorem~\ref{t:GR} helps to construct a simple example of a non-compact $\hTS$-closed linear topological semilattice: just take the semilattice $X=\w\cup\{\w\}$ endowed with the semilattice operation of minimum and the topology $\tau=\{U\subset X:\w\in U\;\Ra\;|U\setminus 2\w|<\w\}$ where $2\w=\{2n:n\in\w\}$.
\end{remark}


\end{document}